\documentclass{amsart}

\newtheorem{theorem}{Theorem}[section]
\newtheorem{lemma}[theorem]{Lemma}
\newtheorem{proposition}[theorem]{Proposition}
\newtheorem{corollary}[theorem]{Corollary}
\newtheorem{question}[theorem]{Question}

\theoremstyle{definition}

\def\oneconst{c}

\usepackage{amscd,amssymb}

\begin{document}

\title[Noether bound for invariants in  relatively free algebras]
{Noether bound for invariants\\
in relatively free algebras}

\author[M\'aty\'as Domokos and Vesselin Drensky]
{M\'aty\'as Domokos and Vesselin Drensky}
\address{MTA Alfr\'ed R\'enyi Institute of Mathematics,
Re\'altanoda utca 13-15, 1053 Budapest, Hungary}
\email{domokos.matyas@renyi.mta.hu}
\address{Institute of Mathematics and Informatics,
Bulgarian Academy of Sciences,
Acad. G. Bonchev Str., Block 8,
1113 Sofia, Bulgaria}
\email{drensky@math.bas.bg}

\thanks{This project was carried out in the framework of the exchange program
between the Hungarian and Bulgarian Academies of Sciences.
It was partially supported by Grant K101515 of the Hungarian  National Foundation for Scientific Research (OTKA)
and by Grant I02/18 of the Bulgarian National Science Fund.}

\subjclass[2010]{16R10; 13A50; 15A72; 16W22.}

\keywords{relatively free associative algebras, invariant theory of finite groups, noncommutative invariant theory,
Noether bound.}

\begin{abstract}
Let $\mathfrak R$ be a weakly noetherian  variety of unitary associative algebras (over a field $K$ of characteristic 0),
i.e., every finitely generated algebra from $\mathfrak{R}$ satisfies the ascending chain condition for two-sided ideals.
For a finite group $G$ and a $d$-dimensional $G$-module $V$ denote  by $F({\mathfrak R},V)$
the relatively free algebra in $\mathfrak{R}$ of rank $d$ freely generated by
the  vector space $V$. It is proved that  the subalgebra $F({\mathfrak R},V)^G$ of $G$-invariants is generated by elements of degree at most
$b(\mathfrak{R},G)$ for some explicitly given number $b(\mathfrak{R},G)$ depending only on the variety $\mathfrak{R}$ and the group $G$ (but not on $V$).
This generalizes the classical result of Emmy Noether
stating that the algebra of commutative polynomial invariants $K[V]^G$
is generated by invariants of degree at most $\vert G\vert$.
\end{abstract}

\maketitle

\section{Introduction}

We fix a base field $K$ of characteristic 0. Throughout the paper  $V=V_d$ denotes a $K$-vector space  of dimension $d\geq 2$
with basis $X_d=\{x_1,\ldots,x_d\}$. We consider the polynomial algebra $K[V]$
and the free unitary associative algebra
$K\langle V\rangle=K\langle x_1,\ldots,x_d\rangle$ freely generated by $V$ over $K$.
The canonical action of the general linear group $GL(V)$
on $V$ is extended diagonally
on $K[V]$ and $K\langle V\rangle$ by the rule
\[
g(x_{j_1}\cdots x_{j_n})=g(x_{j_1})\cdots g(x_{j_n}),
\]
where $g\in GL(V)$ and the monomials $x_{j_1}\cdots x_{j_n}$ belong to $K[V]$ or to $K\langle V\rangle$.
So $K\langle V\rangle $ is the tensor algebra of $V$, whereas $K[V]$ is the symmetric tensor algebra of $V$.
Note that in commutative invariant theory $K[V]$ is usually identified with the algebra of polynomial functions on the dual space of $V$.
One of the ways to develop noncommutative invariant theory is to study
invariants of subgroups of $GL(V)$ acting on factor algebras $K\langle V\rangle/I$, where the ideal $I$ of $K\langle V\rangle$ is stable under
the action of $GL(V)$. The most attractive ideals for this purpose are the T-ideals, i.e., the ideals invariant under all endomorphisms
of $K\langle V\rangle$. Every T-ideal coincides with the ideal $\text{Id}(R,V)$ of the polynomial identities in $d$ variables
of a unitary algebra $R$. Then $K\langle V\rangle/\text{Id}(R,V)$ is the relatively free algebra $F({\mathfrak R},V)$ of rank $d$ in the variety
${\mathfrak R}=\text{var}(R)$ of unitary algebras generated by $R$.
Note that $\text{Id}(R,V)$ is necessarily contained in the commutator ideal of $K\langle V\rangle$, hence we have the natural surjections
\begin{equation}\label{eq:surjections}
K\langle V\rangle \twoheadrightarrow F({\mathfrak R},V)\twoheadrightarrow K[V] .
\end{equation}
In the sequel we assume that $R$ is a PI-algebra, i.e., $\text{Id}(R,V)\not=0$.
For a background on PI-algebras and varieties of algebras, see e.g., \cite{DF} or \cite{GZ}, and on noncommutative invariant theory
the surveys \cite{F} and \cite{D3}.
Recall that a polynomial $f(V)=f(X_d)\in K\langle V\rangle$ is a polynomial identity for the algebra $R$
if $f(r_1,\ldots,r_d)=0$ for all $r_1,\ldots,r_d\in R$. The class $\mathfrak R$ of all algebras satisfying a given system
of polynomial identities is called a {\it variety}. The {\it variety $\mathfrak R$ is generated by} $R$
if $\mathfrak R$ has the same polynomial identities as $R$. Then we write $\text{Id}({\mathfrak R},V)=\text{Id}(R,V)$.
The action of the group $GL(V)$ on $K\langle V\rangle$ induces
an action on the relatively free algebra $F({\mathfrak R},V)$, and the surjections \eqref{eq:surjections} are $GL(V)$-equivariant.

Now let $G$ be a finite group. We say that $V$ is a {\it $G$-module}  if we are given a representation (i.e., a group homomorphism) $\rho:G\to GL(V)$.
We shall suppress $\rho$ from the notation, and write $gv:=(\rho(g))(v)$ for $g\in G$, $v\in V$, and similarly $gf:=(\rho(g))(f)$ for $f\in F(\mathfrak{R},V)$.
Moreover, we shall study the {\it algebra of invariants}
\[
F({\mathfrak R},V)^G=\{f \in F({\mathfrak R},V)\mid gf=f \text{ for all } g\in G\}.
\]
Since the characteristic of $K$ is assumed to be zero, the group $G$ acts completely reducibly on $K\langle V\rangle$,
hence the $G$-equivariant $K$-algebra surjections
in \eqref{eq:surjections} restrict to $K$-algebra surjections
\begin{equation}\label{eq:restricted_surjections}
K\langle V\rangle^G \twoheadrightarrow F({\mathfrak R},V)^G\twoheadrightarrow K[V]^G .
\end{equation}
Our starting point is the following classical fact:
\begin{theorem}[Emmy Noether \cite{N}] \label{thm:noether} Let $G$ be a finite group and $V$ a $G$-module. Then the following holds:
\begin{itemize}
\item[(i)] The algebra $K[V]^G$ is finitely generated.
\item[(ii)] The algebra  $K[V]^G$ is generated by its elements of degree at most $|G|$.
\end{itemize}
\end{theorem}
Of course, (ii) implies (i). We mention that (i) holds also in the modular case
(i.e., when the characteristic of the base field divides the group order), whereas (ii) does not.
In view of (ii) it  makes sense to introduce the numbers
\[
\beta(G,V)=\min\{m\mid K[V]^G \text{ is generated by invariants of degree }\leq m\}
\]
and
\[
\beta(G)={\max}_V\{\beta(G,V)\mid V\mbox{ is a }G\mbox{-module}\}.
\]
The latter is called the {\it Noether number of} $G$, and Theorem~\ref{thm:noether} (ii) says that $\beta(G)\leq\vert G\vert$.
The exact value of $\beta(G)$ is known in few cases only.
Barbara Schmid \cite{Sch} showed that $\beta(G)=\beta(G,V_{\text{reg}})$, where $V_{\text{reg}}$ is the regular $\vert G\vert$-dimensional $G$-module.
It is known that $\beta(G)=\vert G\vert$ for $G$ cyclic. Domokos and Heged\H{u}s \cite{DH} proved that if $G$ is not cyclic then
$\beta(G)\leq(3/4)\vert G\vert$ and this bound is exact because is reached for the Klein four-group and for the quaternion group of order 8.
Cziszter and Domokos \cite{CD2, CD3} showed that the only noncyclic groups $G$ with $\beta(G)\geq(1/2)\vert G\vert$
are the groups with a cyclic subgroup of index 2 and four more sporadic exceptions -- ${\mathbb Z}_3\times{\mathbb Z}_3$,
${\mathbb Z}_2\times{\mathbb Z}_2\times{\mathbb Z}_2$, the alternating group $A_4$, and the binary tetrahedral group
$\tilde{A}_4$. In particular, they proved that $\beta(G)-(1/2)\vert G\vert=1$ or 2 for groups $G$ with cyclic subgroup of index 2.
See also \cite{C} and \cite{CDG} for further information on $\beta(G)$.

Note that in the special case when $\mathfrak{R}$ is the variety of commutative algebras, we have $F(\mathfrak{R},V)=K[V]$ and
hence $F(\mathfrak{R},V)^G=K[V]^G$. Taking the point of view of universal algebra, we may fix a variety $\mathfrak{R}$
(larger than the variety of commutative algebras), and
look for possible analogues of Theorem~\ref{thm:noether}  for the variety $\mathfrak{R}$.
Kharchenko \cite{Kh} characterized the varieties $\mathfrak{R}$ for which $F(\mathfrak{R},V)^G$ is finitely generated for all finite groups $G$ and $G$-modules $V$.
More precisely, he showed that if $F(\mathfrak{R},V_2)^G$ is finitely generated for all finite subgroups $G$ of $GL(V_2)$,
then the variety $\mathfrak{R}$ is {\it weakly noetherian}
(i.e., every finitely generated algebra from $\mathfrak{R}$ satisfies the ascending chain condition for two-sided ideals),
and conversely, if $\mathfrak{R}$ is weakly noetherian, then $F(\mathfrak{R},V)^G$ is finitely generated for all finite $G$ and $G$-module $V$.
By analogy with the definition of $\beta(G,V)$ and $\beta(G)$, given a weakly noetherian variety $\mathfrak{R}$ we introduce the numbers
\[
\beta(G,{\mathfrak R},V)=\min\{m\mid F({\mathfrak R},V)^G \text{ is generated by invariants of degree }\leq m\},
\]
\[
\beta(G,{\mathfrak R})={\sup}_V\{\beta(G,{\mathfrak R},V)\mid V\mbox{ is a }G\mbox{-module}\}.
\]
The following natural question arises:
\begin{question}\label{question} Let $\mathfrak{R}$ be a weakly noetherian variety of unitary associative $K$-algebras.
\begin{enumerate}
\item Is $\beta(G,\mathfrak{R})$ finite for all finite groups $G$?
\item If the answer to (1) is yes, find an upper bound for $\beta(G,\mathfrak{R})$ in terms of $|G|$ and some numerical invariants of $\mathfrak{R}$.
\end{enumerate}
\end{question}
The main result of the present paper is a positive answer to the above questions: in Theorem~\ref{main theorem 2}
we give an explicit bound for $\beta(G,\mathfrak{R})$ in terms of $|G|$ and some quantities associated to $\mathfrak{R}$.

The paper is organized as follows.
In Section~\ref{sec:aux} we present necessary facts from the theory of polynomial identities and invariant theory.
First we collect several characterizations of weakly noetherian varieties in Theorem~\ref{finite generation for all finite G}. Next
we recall the theorem of Latyshev \cite{La} that if a finitely generated PI-algebra $R$ satisfies a nonmatrix polynomial identity,
then the commutator ideal $C(R)=R[R,R]R$ of $R$ is nilpotent. We shall also need the Nagata-Higman theorem \cite{Na, H}
that (nonunitary) nil algebras of bounded nil index are nilpotent. We continue with some lemmas about graded modules
and commutative invariant theory and deduce consequences for the noncommutative case.
Section~\ref{sec:main} contains our main results, throughout this section we fix a weakly noetherian variety $\mathfrak{R}$.
In Theorem~\ref{main theorem 1} we provide an upper bound for $\beta(G,\mathfrak{R},V)$ in terms of $\beta(G)$,
the degree of an identity  of the form \eqref{identity in three variables}
(see Theorem~\ref{finite generation for all finite G} (viii)) satisfied by $\mathfrak{R}$,
and the index of nilpotency of the commutator ideal of $F({\mathfrak R},V)$.
In particular, this gives a  new and effective proof of the result of Kharchenko \cite{Kh} that
the weak noetherianity of $F({\mathfrak R},V)$ implies the finite generation of the algebra of $G$-invariants $F({\mathfrak R},V)^G$,
i.e., for the implication $\{\text{(iii) and (viii)}\}\Rightarrow\text{(i)}$ in Theorem \ref{finite generation for all finite G}.
However, this result does not yet answer Question~\ref{question} (i), since  no noncommutative analogues of the result
$\beta(G)=\beta(G,V_{\text{reg}})$ is available, and if $\mathfrak R$ cannot be generated by a finitely generated algebra $R$,
then the class of nilpotency of the commutator ideal of $F({\mathfrak R},V)$ depends on the dimension of the vector space $V$
(and consequently the bound in Theorem~\ref{main theorem 1} tends to infinity as the dimension of $V$ grows).
Using a  different strategy we prove an upper bound on $\beta(\mathfrak{R},G)$ in Theorem~\ref{main theorem 2},
which depends only on  $\vert G\vert$, $\beta(G)$, the degree of an identity of the form \eqref{identity in three variables}
satisfied by $\mathfrak{R}$ and on  the class of nilpotency of nil algebras of index $\ell$,
where $\ell$ is the class of nilpotency of the commutator ideal of one
relatively free algebra, namely $F({\mathfrak R},V_{\vert G\vert})$.
Theorem~\ref{main theorem 2} is independent from Theorem~\ref{main theorem 1}. Although this settles Question~\ref{question},
for low dimensional $G$-modules $V$ the  bound for $\beta(G,\mathfrak{R},V)$ provided in
Theorem~\ref{main theorem 1} has smaller value than the general bound on $\beta(G,\mathfrak{R})$ from Theorem~\ref{main theorem 2}.
Finally, in Theorem~\ref{main theorem 3} we improve the bound on $\beta(G,\mathfrak{R})$ for an abelian group $G$,
and give one which  depends only on the degree of the polynomial identity
(\ref{identity in three variables}) satisfied by $\mathfrak R$ and the order of $G$, but  does not depend on the class of nilpotency
of the commutator ideal of any of the relatively free algebras in $\mathfrak R$.


\section{Auxiliaries}\label{sec:aux}

Although $F({\mathfrak R},V)$ shares many properties of polynomial algebras, it has turned out that the finite generation
of $F({\mathfrak R},V)^G$ for all finite $G$ forces very strong restrictions on $\mathfrak R$.
Below we summarize the known results, see the survey articles \cite{D3, KS}. Recall that for an associative algebra $R$
\[
[u,v]=u(\text{ad}v)=uv-vu,\quad u,v\in R,
\]
is the commutator of $u$ and $v$. Our commutators are left normed, i.e.,
\[
[u_1,\ldots,u_{n-1},u_n]=[[u_1,\ldots,u_{n-1}],u_n],\quad u_1,\ldots,u_{n-1},u_n\in R,\quad n\geq 3.
\]
First we define a sequence of PI-algebras $R_k$,
$k = 1,2,\ldots$. Let $D_k = K[t]/(t^k)$ and let
\[
R_k = \left(\begin{matrix}
D_k&tD_k\\
tD_k&D_k\\
\end{matrix}\right)
\subset M_2(D_k),
\]
where $M_2(D_k)$ is the $2 \times 2$ matrix algebra with entries
from $D_k$. These algebras appear in \cite{D1} and ``almost'' describe
the T-ideals containing strictly the T-ideal $\text{Id}(M_2(K),V)$.
(Another description of those T-ideals was given by Kemer \cite{K2}.)

\begin{theorem}\label{finite generation for all finite G}
Let $\mathfrak R$ be a variety of algebras. The following
conditions on $\mathfrak R$ are equivalent. If some of them is
satisfied for some vector space $V_{d_0}$ of dimension $d_0 \geq 2$, then all of them hold for all $d$-dimensional vector spaces $V$, $d\geq 2$:

{\rm (i)} The algebra $F({\mathfrak R},V)^G$ is finitely generated for every
finite subgroup $G$ of $GL(V)$.

{\rm (ii)} The algebra $F({\mathfrak R},V)^{\langle g\rangle}$ is finitely
generated, where $g \in GL(V)$ is a matrix of finite multiplicative
order with at least two eigenvalues (or characteristic roots) of
different order.

{\rm (iii)} The algebra $F({\mathfrak R},V)$ is {\bf weakly noetherian}, i.e.,
satisfies the ascending chain condition for two-sided ideals.

{\rm (iv)} Let $S$ be an algebra satisfying all the polynomial
identities of $\mathfrak R$ (i.e., $S \in \mathfrak R$) and generated by $d$
elements $s_1,\ldots,s_d$. Then $S$ is {\bf finitely presented} as
a homomorphic image of $F({\mathfrak R},V)$, i.e., the kernel of the canonical
homomorphism $F({\mathfrak R},V) \longrightarrow S$ defined by
$x_i \longrightarrow s_i$, $i = 1,\ldots,d$, is a finitely generated
ideal of $F({\mathfrak R},V)$.

{\rm (v)} If $S$ is a finitely generated algebra from $\mathfrak R$,
then $S$ is {\bf residually finite}, i.e., for every
nonzero element $s \in S$ there exist a finite dimensional algebra $D$ and a
homomorphism $\varphi:S \longrightarrow D$ such that $\varphi(s) \not=0$.

{\rm (vi)} If $S$ is a finitely generated algebra from $\mathfrak R$,
then $S$ is {\bf representable by matrices}, i.e., there
exist an extension $L$ of the base field $K$ and an integer $k$
such that $S$ is
isomorphic to a subalgebra of the $K$-algebra $M_k(L)$ of all
$k\times k$ matrices with entries from $L$.

{\rm (vii)} If the base field $K$ is countable, then  the set of
pairwise non-isomorphic homomorphic images of $F({\mathfrak R},V)$ is
countable.

{\rm (viii)} For some $n\geq 2$ the variety $\mathfrak R$ satisfies a polynomial identity of the form
\begin{equation}\label{identity in three variables}
\begin{array}{c}
f(x_1,x_2,x_3)\displaystyle{=x_2x_1^nx_3 +\gamma x_3x_1^nx_2
+\sum_{i+j>0}\alpha_{ij}x_1^ix_2x_1^{n-i-j}x_3x_1^j}\\
\displaystyle{+\sum_{i+j>0}\beta_{ij}x_1^ix_3x_1^{n-i-j}x_2x_1^j= 0},\quad\alpha_{ij}, \beta_{ij},\gamma \in K.\\
\end{array}
\end{equation}

{\rm (ix)} For some $n\geq 2$ the variety $\mathfrak R$ satisfies a polynomial identity of the form
\[
x_2x_1^nx_2 +
\sum_{i+j>0}\alpha_{ij}x_1^ix_2x_1^{n-i-j}x_2x_1^j = 0,\quad
\alpha_{ij} \in K.
\]

{\rm (x)} The variety $\mathfrak R$ satisfies the polynomial identity
\[
[x_1,x_2,\ldots,x_2]x_3^n[x_4,x_5,\ldots,x_5] = 0
\]
for sufficiently long commutators and $n$ large enough.

{\rm (xi)} The variety $\mathfrak R$ satisfies
the polynomial identity
\[
[x_1,\ldots,x_n]x_{n+1}\ldots x_{2n}[x_{2n+1},\ldots,x_{3n}] = 0
\]
for some positive integer $n$.

{\rm (xii)} The variety $\mathfrak R$ satisfies a polynomial identity which
does not follow from the polynomial identities
\[
[x_1,x_2][x_3,x_4][x_5,x_6] = 0,\quad
[[x_1,x_2][x_3,x_4],x_5] = 0,\quad
s_4(x_1,x_2,x_3,x_4) = 0.
\]

{\rm (xiii)} The T-ideal $\text{\rm Id}({\mathfrak R},V)$ is not contained in the T-ideal
$\text{\rm Id}(R_3,V)$ of the algebra $R_3$ defined above.
\end{theorem}

The equivalence of (i) and (iii) was established by Kharchenko \cite{Kh},
of (iii), (viii), and (ix) by L'vov \cite{Lv},
of (iii), (v), (vi), (x) and (xi) (for finitely generated algebras $R\in\mathfrak R$) by Anan'in \cite{A},
of (ix) and (xii) by Tonov \cite{T}, of (ii),
(viii) and (xiii) by Drensky \cite{D2}. The equivalence of (iii),
(iv) and (vii) is obvious.
The general case of the implication (v) $\Rightarrow$ (ix) is due to Kemer \cite{K1}
who showed that associative algebras satisfying the Engel identity are Lie nilpotent.
(The theorem that Lie algebras with the Engel identity are nilpotent was proved by
Zelmanov \cite{Z}.)
We want to mention that the study of representable
algebras begins with the paper by Malcev \cite{M}.
The condition (ii) is a generalization of the following result of
Fisher and Montgomery \cite{FM}. If
$\text{Id}({\mathfrak R},V) \subseteq \text{Id}(M_2(K),V)$, $g \in GL(V)$, $g^n = 1$, and $g$ has
at least two characteristic roots of different multiplicative
order, then the algebra of invariants $F({\mathfrak R},V)^{\langle g\rangle}$
is not finitely generated. The condition (ii) gives
a simple criterion to check the equivalent conditions of Theorem \ref{finite generation for all finite G}.
It is sufficient to choose $d = 2$ and
\[
g = \left(\begin{matrix}
-1&0\\
0&1\\
\end{matrix}\right).
\]
If $F({\mathfrak R},V_2)^{\langle g\rangle}$ is finitely generated for the $2$-dimensional vector space $V_2$, then all the
assertions (i) -- (xiii) hold for $F({\mathfrak R},V)$ and $\text{Id}({\mathfrak R},V)$ for any $d$-dimensional vector space $V$, $d \geq 2$.

Several parts of the proof of Theorem \ref{finite generation for all finite G} depend on two important results in the theory of PI-algebras,
which play crucial role also in the proofs of our results.

\begin{theorem}[Latyshev \cite{La}]\label{theorem of Latyshev}
Let $R$ be a finitely generated algebra which satisfies a nonmatrix polynomial identity,
i.e., a polynomial identity which does not hold for the $2\times 2$ matrix algebra $M_2(K)$. Then $R$ satisfies the identity
\[
[x_1,x_2]\cdots[x_{2k+1},x_{2k+2}]=0
\]
for some $k$. Equivalently, the commutator ideal $C(R)=R[R,R]R$ of $R$ is nilpotent of class $k+1$.
\end{theorem}

The next result we need is the Nagata-Higman theorem \cite{Na, H} which is one of the milestones of PI-theory.
(More precisely, it should be called the Dubnov-Ivanov-Nagata-Higman theorem, established first by Dubnov and Ivanov \cite{DuI} in 1943,
and then independently by Nagata \cite{Na} in 1953. The proof of Higman \cite{H} from 1956
covers also the case of nil algebras over fields of positive characteristic $p>n$.)
Since we shall consider
nil and nilpotent algebras which are without unit, in the next lines we
work with nonunitary algebras and, in particular, with the free nonunitary algebra
$K\langle V_{\infty}\rangle_+$ on the vector space $V_{\infty}$ of countable dimension.

\begin{theorem}\label{Nagata-Higman}
Let $R$ be an associative algebra without unit which is nil of bounded index, i.e., $R$ satisfies
a polynomial identity $x^n=0$. Then $R$ is nilpotent, i.e., it satisfies the identity $x_1\cdots x_N=0$ for some $N$.
\end{theorem}

Let $\nu(n)$ be the minimal positive integer such that the polynomial  $x_1\cdots x_{\nu(n)}$
belongs to the T-ideal generated by $x^n$. We call $\nu(n)$ the {\it Nagata-Higman number} for nil algebras of index $n$.
The upper bound $\nu(n)\leq 2^n-1$ is given in the proof of Higman \cite{H}.
The best known bounds for $\nu(n)$
\[
\frac{n(n+1)}{2}\leq \nu(n) \leq n^2
\]
are due respectively to Kuz'min \cite{Ku} (see also \cite[Theorem 6.2.7, page 85]{DF}) and Razmyslov \cite{R}.
Kuz'min \cite{Ku} conjectured that
\[
\nu(n)=\frac{n(n+1)}{2},
\]
and this is confirmed for $n \leq 4$, with partial results for $n=5$, see the comments in \cite[page 79]{DF}.

The Nagata-Higman theorem has a precision which is contained in the proof of Higman \cite{H},
see also \cite[Theorem 6.1.2 (ii), pages 75-77]{DF}.

\begin{proposition}\label{Nagata-Higman as vector space}
The T-ideal generated by $x^n$ in
the free nonunitary algebra $K\langle V_{\infty}\rangle_+$
coincides with the vector space spanned by
all $n$-th powers. In particular,
for $m\geq \nu(n)$
the monomial $x_1\cdots x_m$ has the form
\[
x_1\cdots x_m=\sum\alpha_uu^n
\]
for some $\alpha_u\in K$ and $u\in K\langle V_m\rangle_+$.
\end{proposition}

We continue with some facts about graded modules. Suppose that $\displaystyle R=\bigoplus_{d=0}^\infty R_d$ is a graded $K$-algebra with $R_0=K$,
on which the finite group $G$ acts via graded $K$-algebra automorphisms. Let $\displaystyle M=\bigoplus_{d=0}^\infty M_d$
be a finitely generated graded $R$-module (left module), where each homogeneous component $M_d$ is a $G$-module.
Suppose that the $R$-module structure of $M$ is compatible with the $G$-action on $R$ and $M$, i.e., for $g\in G$, $r\in R$, and $m\in M$ we have
$g(rm)=(gr)(gm)$. Write $\beta(R)$ for the minimal $n$ such that $R$ is generated as a $K$-algebra by homogeneous elements of degree at most $n$, and denote by
$\gamma(M,R)$ the minimal $n$ such that the $R$-module $M$ is generated by homogeneous elements of degree at most $n$.
A set of homogeneous elements generates $R$ as a $K$-algebra if and only if they generate $R_+$ (the sum of the positive degree homogeneous components of $R$)
as an $R$-bimodule (i.e., as an ideal). If $S$ is a finitely generated $K$-subalgebra of $R$,  and $R$ is a finitely generated  $S$-module,
then $M$ is a finitely generated $S$-module, and we have the obvious inequality
\begin{equation}\label{eq:gamma(M,S)}
\gamma(M,S)\le \gamma(M,R)+\gamma(R,S).
\end{equation}
Moreover, we have the inequality
\begin{equation}\label{eq:reynolds}
\gamma(M^G,R^G)\le\gamma(M,R^G).
\end{equation}
Indeed, the Reynolds operator $\rho:M\to M^G$, $\displaystyle m\mapsto\frac{1}{|G|}\sum_{g\in G}gm$ is a surjective $R^G$-module homomorphism,
and therefore $\rho$ maps a homogeneous $R^G$-module generating system of $M$ to an $R^G$-module generating system of $M^G$.
Since $\rho$ preserves the degrees, \eqref{eq:reynolds} follows.
Next we reformulate a result from \cite{CD1}:

\begin{lemma}\label{lemma:CD1}
We have the inequality $\gamma(K[V],K[V]^G)\le \beta(G)-1$.
\end{lemma}

\begin{proof}
Lemma 3.1 of \cite{CD1} gives that there exists an irreducible $G$-module $U$
such that
\[
\beta(G,V\oplus U)\geq\gamma(K[V],K[V]^G)+1.
\]
Since $\beta(G,V\oplus U)$ is trivially bounded  by $\beta(G)$, we obtain the desired inequality.
\end{proof}

\begin{corollary}\label{general case for bound for modules}
Let $M$ be a finitely generated graded $R$-module as above, where $R=K[V]$ for some $G$-module $V$.
Then $M$ and $M^G$ are finitely generated $R^G$-modules, and we have the inequalities
\[\gamma(M^G,R^G)\le \gamma(M,R^G)\le \gamma(M,R)+\beta(G)-1.\]
\end{corollary}
\begin{proof}
By \eqref{eq:reynolds} and \eqref{eq:gamma(M,S)} we have $\gamma(M^G,R^G)\le \gamma(M,R^G)\le \gamma(M,R)+\gamma(R,R^G)$.
By Lemma~\ref{lemma:CD1} we have $\gamma(R,R^G)\le \beta(G)-1$.
\end{proof}

\begin{lemma}\label{generators of invariants in X and Y}
Let $V$ and $W$ be isomorphic as $G$-modules. Identifying $K[V]\otimes K[W]$ and $K[V\oplus W]$,
we have
\[\gamma(K[V\oplus W]^G,K[V]^G\otimes K[W]^G)\leq 2(\beta(G)-1).\]
\end{lemma}
\begin{proof}
By Lemma~\ref{lemma:CD1} the $K[V]^G$-module $K[V]$ is generated by homogeneous polynomials $w_1(V),\ldots,w_s(V)$ of degree $\leq \beta(G)-1$. The products
$\{w_a(V)w_b(W)\mid a,b=1,\dots,s\}$ generate $K[V\oplus W]$ as an $R=K[V]^G\otimes K[W]^G$-module, thus
$\gamma(K[V\oplus W],R)\le 2(\beta(G)-1)$.   Now apply \eqref{eq:reynolds} for $M=K[V\oplus W]$ viewed as an $R$-module
(where $G$ acts trivially on $R$, i.e., $R=R^G$).
\end{proof}

\begin{lemma}\label{lemma from commutative invariants}
Let $V$ be a $G$-module  and let $\pi:K[V]\to R$ be a surjective $K$-algebra homomorphism of $K[V]$ onto $R$ such that
$\ker(\pi)\subset K[V]_+$ and $G(\ker(\pi))=\ker(\pi)$.
Then $R_+^{\beta(G)}$ is contained in the ideal $RR_+^G$ of $R$ generated by the subalgebra $R_+^G$
consisting of the elements fixed by the induced action of $G$ on $R_+=\pi(K[V]_+)$.
\end{lemma}

\begin{proof}
First we shall establish the lemma for $R=K[V]$. By the graded Nakayama lemma (see for example Lemma 3.5.1 in \cite{DK})
$\gamma(K[V],K[V]^G)$ coincides with the top degree of the factor space $K[V]/K[V]K[V]^G_+$ (inheriting a grading from $K[V]$). Therefore
$K[V]_+^{\beta(G)}\subseteq K[V]K[V]^G_+$ holds by Lemma~\ref{lemma:CD1}. Applying $\pi$ to this inclusion we get
$R_+^{\beta(G)}=\pi(K[V]_+)^{\beta(G)}=\pi(K[V]_+^{\beta(G)})\subseteq \pi(K[V]K[V]^G_+)=RR^G_+$.
\end{proof}

\begin{corollary}\label{ideal of squares of invariants}
Let $V$ be a $G$-module and let $\pi:K\langle V\rangle\to R$ be a surjective $K$-algebra homomorphism of $K\langle V\rangle$ onto $R$ such that
$\ker(\pi)\subset K\langle V\rangle_+$ and $G(\ker(\pi))=\ker(\pi)$ (so the action of $G$ on $K\langle V\rangle$
induces a $G$-action on $R$ via $K$-algebra automorphisms).
If the commutator ideal of $R$ is nilpotent of class $\ell+1$,
then $R_+^{2\beta(G)(\ell+1)}$ is contained in the ideal $R(R_+^G)^2R$ of $R$ generated by the square of the subalgebra $R_+^G$
of  $G$-fixed elements of $R_+=\pi(K\langle V\rangle_+)$.
In particular, if $u\in R_+$, then
\[
u^{2\beta(G)(\ell+1)}\in R(R_+^G)^2R.
\]
\end{corollary}
\begin{proof}
Let $\bar{R}=R/C$, where $C$ is the commutator ideal of $R$.
Since $C$ is stable under the action of $G$, there is an induced $G$-action on $R/C$ via $K$-algebra automorphisms.
Complete reducibility of the $G$-action on $R$ implies that  $R^G/C^G=(R/C)^G$.
Then, by Lemma \ref{lemma from commutative invariants}, $(\bar{R}_+)^{\beta(G)}\subseteq \bar{R}\bar{R}_+^G$
and hence  $(\bar{R}_+)^{2\beta(G)}\subseteq \bar{R}(\bar{R}_+^G)^2$.
Therefore $R_+^{2\beta(G)}\subset R(R_+^G)^2+C$.
Since $C^{\ell+1}=(0)$, we have that $(R(R_+^G)^2+C)^{\ell+1}\subseteq R(R_+^G)^2R$, implying
$R_+^{2\beta(G)(\ell+1)}\subseteq R(R_+^G)^2R$.
\end{proof}

\begin{lemma}\label{action of abelian groups}
Let $G$ be a finite abelian group, and denote by  $G^{\ast}$ the group of the characters (i.e., homomorphisms $G\to K^\times$) of $G$.
Suppose that  $G$ acts on $V$  as a group of diagonal matrices, i.e.,
the action on the basis $X_d$ of $V$ is given by
\[
g(x_i)=\chi_i(g)x_i,\quad \chi_i\in G^{\ast}, \quad i=1,\ldots,d.
\]
Then for any $\vert G\vert$ words $u_1,\ldots, u_{\vert G\vert}\in K\langle V\rangle$,
\[
u_i=x_{ij_1}\cdots x_{ij_{s_i}},\quad x_{ij}\in X_d, \quad i=1,\ldots,\vert G\vert,
\]
the product $u_1\cdots u_{\vert G\vert}$ contains a $G$-invariant subword of the form $u_{i+1}\cdots u_j$, $1\leq i< j\leq\vert G\vert$.
\end{lemma}

\begin{proof}
The group $G$ acts on $u_i$ and on $u_1\cdots u_i$ by the rule
\[
g(u_i)=g(x_{ij_1})\cdots g(x_{ij_{s_i}})=\chi_{ij_1}(g)\cdots \chi_{ij_{s_i}}(g)x_{ij_1}\cdots x_{ij_{s_i}}=\chi^{(i)}(g)u_i,
\]
\[
g(u_1\cdots u_i)=\chi^{(1)}(g)\cdots \chi^{(i)}(g)u_1\cdots u_i.
\]
 Consider the $\vert G\vert +1$ products of characters
\[
\chi^{(0)}=\chi_{\text{\rm id}},\quad
\chi^{(i)}=\chi_1\cdots\chi_i,\quad i=1,\ldots,\vert G\vert.
\]
Since $\vert G^{\ast}\vert\le \vert G\vert$ (with equality if $K$ is algebraically closed),
by the Pigeonhole Principle, there exist two characters $\chi^{(i)}=\chi_1\cdots\chi_i$ and
$\chi^{(j)}=\chi_1\cdots\chi_j$, $0\leq i<j\leq \vert G\vert$ which are equal. Hence $\chi_{i+1}\cdots\chi_j=\chi_{\text{\rm id}}$.
This means that the product $u_{i+1}\cdots u_j$ is $G$-invariant.
\end{proof}


\section{The main results}\label{sec:main}
Till the end of the paper we fix a variety of unitary algebras $\mathfrak R$ satisfying the polynomial identity
(\ref{identity in three variables}) from Theorem \ref{finite generation for all finite G} (viii).
Replacing $x_3$ by $x_1x_3$ in $f(x_1,x_2,x_3)=0$ from (\ref{identity in three variables}) we obtain a multihomogeneous consequence of total degree $n+3$
of the form
\begin{equation}\label{identity of special kind}
h(x_1,x_2,x_3)=x_2x_1^{n+1}x_3+x_1h_1(x_1,x_2,x_3)+h_2(x_1,x_2,x_3)x_1=0
\end{equation}
where $h_1,h_2$ are multihomogeneous of total degree $n+2$.
In (\ref{identity of special kind}) we replace $x_1,x_2,x_3$ by $u\in K\langle V_{\infty}\rangle_+$, $y$, and $z$, respectively, and  obtain
\begin{equation}\label{eq:master}
h(u,y,z)=yu^{n+1}z+uh_1(u,y,z)+h_2(u,y,z)u=0,
\end{equation}
i.e., $yu^{n+1}z$ can be expressed as a linear combination of polynomials starting or finishing with $u$.
Applying Proposition \ref{Nagata-Higman as vector space} we obtain a consequence of (\ref{identity in three variables})
of the form
\begin{equation}\label{Nagata-Higman type identity for R}
\begin{array}{c}
h'(x_1,\ldots,x_{\nu},y,z)=yx_1\cdots x_{\nu}z+\sum_{i=1}^{\nu}(x_iv_i'(x_1,\ldots,\widehat{x_i},\ldots,x_{\nu},y,z)\\
\\
+v_i''(x_1,\ldots,\widehat{x_i},\ldots,x_{\nu},y,z)x_i)=0,\quad \nu=\nu(n+1).\\
\end{array}
\end{equation}
We fix the notation
\[
F=F({\mathfrak R},V)\text{ and } C=C({\mathfrak R},V)=F[F,F]F
\]
for the relatively free algebra on $V$ and its commutator ideal, respectively.
It is well known that $C^p$ modulo $C^{p+1}$ is spanned on the products
\begin{equation}\label{typical element in the power of the commutator ideal}
w=X_d^{a^{(0)}}[x_{i_1},x_{j_1}]X_d^{a^{(1)}}[x_{i_2},x_{j_2}]\cdots X_d^{a^{(p-1)}}[x_{i_p},x_{j_p}]X_d^{a^{(p)}},
\end{equation}
where $X_d^a=x_1^{a_1}\cdots x_d^{a_d}$.
Note that $C^p/C^{p+1}$ is a naturally a $K[V]$-bimodule.
Equivalently, we consider  $C^p/C^{p+1}$ as a module over $K[V\oplus W]\cong K[V]\otimes K[W]$, where
the $G$-module $W$ is isomorphic to $V$ and has a basis $Y_d=\{y_1,\ldots,y_d\}$. The action of
$X_d^bY_d^c=X_d^b\otimes Y_d^c\in K[V]\otimes K[W]$ is defined by
\[
X_d^bY_d^c(w)=X_d^{a^{(0)}+b}[x_{i_1},x_{j_1}]X_d^{a^{(1)}}[x_{i_2},x_{j_2}]\cdots X_d^{a^{(p-1)}}[x_{i_p},x_{j_p}]X_d^{a^{(p)}+c}.
\]

\begin{lemma}\label{bounds of lengths between commutators}
Let $\nu=\nu(n+1)$ be the Nagata-Higman number for the nil algebras of index $n+1$.
Then for $p\geq 1$ the vector space $C^p/C^{p+1}$ is spanned by
the products (\ref{typical element in the power of the commutator ideal})
such that $0\leq a_i^{(q)}\leq n$ and $a_1^{(q)}+\cdots+a_d^{(q)}\leq \nu-1$
for each $a^{(q)}=(a_1^{(q)},\ldots,a_d^{(q)})$, $q=1,\ldots,p-1$.
\end{lemma}

\begin{proof}
Take $w$ as in \eqref{typical element in the power of the commutator ideal}.
By induction on $\displaystyle \sum_{j=1}^{p-1}\sum_{i=1}^da_i^{(j)}$ we shall show that $w+C^{p+1}$ belongs to the subspace of $C^p/C^{p+1}$
spanned by the special elements in the statement. If the above sum is zero, then all $a_i^{(j)}=0$
for $j=1,\dots,p-1$ and $i=1,\dots,d$, so the induction can start.
If $a_i^{(j)}\geq n+1$ for some $i\in\{1,\ldots,d\}$ and some $j\in \{1,\ldots,p-1\}$ then
we have $w=X_d^{a^{(0)}}yu^{n+1}zX_d^{a^{(p)}}$
where
 \[u:=x,\quad
y:=[x_{i_1},x_{j_1}]X_d^{a^{(1)}}
\cdots x_{i-1}^{a_{i-1}^{(j)}}x_i^{a_i^{(j)}-n-1},
\quad z:=x_{i+1}^{a_{i+1}^{(j)}}\cdots X_d^{a^{(p-1)}}[x_{i_p},x_{j_p}].
\]
Applying the identity \eqref{eq:master}
we express $yu^{n+1}z$ (and hence $w$)  as a linear combination of elements of the form \eqref{typical element in the power of the commutator ideal} with
$\displaystyle \sum_{j=1}^{p-1}\sum_{i=1}^da_i^{(j)}$ one less than for $w$.
If $a_1^{(q)}+\cdots+a_d^{(q)}\leq \nu-1$ for some $q\in\{1,\ldots,p-1\}$,
then we can use the polynomial identity (\ref{Nagata-Higman type identity for R}) in a similar vein.
\end{proof}

Now we are ready to prove the main results of our paper.

\begin{theorem}\label{main theorem 1}
Let $V$ be a $d$-dimensional $G$-module ($d\ge 2$) and $\mathfrak R$  a weakly noetherian variety of associative algebras properly
containing the variety of commutative algebras
(so $C\neq 0$).
Then
\[
\beta(G,{\mathfrak R},V)\leq \oneconst(\mathfrak{R},d)+3(\beta(G)-1)
\]
where
\[\oneconst(\mathfrak{R},d)=2(\ell(\mathfrak{R},d)-1)+(\ell(\mathfrak{R},d)-2)\cdot \min\{\nu(n(\mathfrak{R}))-1,(n(\mathfrak{R})-1)d\};\]
here $\ell(\mathfrak{R},d)$ is the nilpotence class of the commutator ideal $C$ of $F$,  $n(\mathfrak{R})$
is the minimal positive integer $n+1$ such that $\mathfrak{R}$ satisfies a polynomial identity of the form
\eqref{identity of special kind} of degree $n+3$, whereas $\beta(G)$ is the Noether number for $G$
and $\nu(n(\mathfrak{R}))$ is the Nagata-Higman number for nil algebras of index $n(\mathfrak{R})$.
\end{theorem}

\begin{proof}
The commutator ideal $C$ of $F$ is stable under the action of $GL(V)$, and hence under the action of its finite
subgroup $G$. Therefore we have an induced action of $G$ on $C^p/C^{p+1}$, and the $G$-invariants in $C^p/C^{p+1}$ can be lifted to $G$-invariants in
$C^p\subset F$. Note  that
$F/C\cong K[V]$ and hence $(F/C)^G\cong K[V]^G$. Moreover,  $(C^p/C^{p+1})^G$ is a $K[V]^G\otimes K[W]^G$-submodule of $C^p/C^{p+1}$ (whose
$K[V\oplus W]$-module structure was introduced in the paragraph preceding Lemma~\ref{bounds of lengths between commutators}).
Every homogeneous system
of generators $\bar u_{p1},\ldots,\bar u_{pr_p}$ of the $K[V]^G\otimes K[W]^G$-module $(C^p/C^{p+1})^G$
can be lifted to  sets of homogeneous $G$-invariants $u_{p1},\ldots,u_{pr_p}\in C^p$ with $\deg(u_{pj})=\deg(\bar u_{pj})$, and
it is straightforward that the elements $u_{pq}$, $p=0,1,\ldots,\ell(\mathfrak{R},d)-1$, $q=1,\ldots,r_p$,
generate $F^G$ as a $K$-algebra (indeed, by induction on $k$ one shows that the images of the elements $u_{pj}$
generate the subalgebra of $G$-invariants in the factor algebra $F/C^k$). Therefore it is sufficient to show that
\begin{equation}\label{eq:Cp}\gamma((C^p/C^{p+1})^G,K[V]^G\otimes K[W]^G)
\leq \oneconst(\mathfrak{R},d)+3(\beta(G)-1)\end{equation}
for $p=0,1,\ldots,\ell(\frak{R},d)-1$.
By  Lemma \ref{bounds of lengths between commutators} we have
\[\gamma(C^p/C^{p+1},K[V\oplus W])
\leq 2p+(p-1)\min(\nu(n(\mathfrak{R}))-1,(n(\mathfrak{R})-1)d) \leq  \oneconst(\mathfrak{R},d)
\]
(the latter inequality is immediate from $p\le \ell(\mathfrak{R},d)-1$).
It follows by Corollary \ref{general case for bound for modules} that
\[\gamma((C^p/C^{p+1})^G,K[V\oplus W]^G)\leq
\oneconst(\mathfrak{R},d)+(\beta(G)-1).
\]
Applying \eqref{eq:gamma(M,S)} with $M=(C^p/C^{p+1})^G$, $R=K[V\oplus W]^G$, $S=K[V]^G\otimes K[W]^G$, and using that
$\gamma(R,S)\leq 2(\beta(G)-1)$ by Lemma \ref{generators of invariants in X and Y}, we conclude that
\[\gamma((C^p/C^{p+1})^G,K[V]^G\otimes K[W]^G)\leq \gamma((C^p/C^{p+1})^G,K[V\oplus W]^G)+2(\beta(G)-1).\]
Combining the above two inequalities we obtain the desired inequality \eqref{eq:Cp}
which completes the proof of the theorem.
\end{proof}

\begin{theorem}\label{main theorem 2}
Let $\mathfrak{R}$ be a weakly noetherian variety of associative algebras properly containing the variety of commutative algebras and $G$ a finite group.
Then
\[
\beta(G,{\mathfrak R})\leq (\nu(n(\mathfrak{R}))-1) \cdot  \nu\left(2\beta(G)\ell(\mathfrak{R},|G|)\right)-1
\]
where $n(\mathfrak{R})$ and  $\ell(\mathfrak{R},|G|)$ are the same as in Theorem~\ref{main theorem 1}.
\end{theorem}

\begin{proof}
Take a $G$-module $V$ and consider $F=F(\mathfrak{R},V)$.
Let $u\in F_+$ and let $W$ be the vector subspace of $F_+$ spanned by $\{g(u)\mid g\in G\}$.
Clearly $W$ is a $G$-submodule of $F$. Hence the (unitary) subalgebra $R$ of $F$ generated by
$W$ is also a $G$-submodule of
$F$ and $R_+^G\subseteq F_+^G$. Since $R$ is generated by at most $\vert G\vert$ elements,
it is a homomorphic image of the relatively free algebra
$F({\mathfrak R},V_{\vert G\vert})$ of rank $\vert G\vert$ in $\mathfrak R$.
Hence the commutator ideal $C(R)$ of $R$ satisfies $C(R)^{\ell(\mathfrak{R},\vert G\vert)}=0$.
On the other hand, the identity map $W\to W$ extends to a $G$-equivariant $K$-algebra surjection $\pi: K\langle W\rangle\to R$.
By Corollary \ref{ideal of squares of invariants},
\[
u^{2\beta(G)\ell(\frak{R},\vert G\vert)}\in R(R_+^G)^2R\subseteq F(F_+^G)^2F.
\]
Hence the factor algebra $F_+/F(F_+^G)^2F$
is nil of index $2\beta(G)\ell(\frak{R},\vert G\vert)$.
Let $\xi=\nu(2\beta(G)\ell(\frak{R},\vert G\vert))$. By Proposition \ref{Nagata-Higman as vector space},
for every $u_1,\ldots,u_{\xi}\in F_+$, the product of their images $\bar{u_i}$ in
$F_+/F(F_+^G)^2F$ satisfies
$\bar{u}_1\cdots\bar{u}_{\xi}=\bar{0}$, i.e.,
\[
u_1\cdots u_{\xi}\in F(F_+^G)^2F.
\]
Hence $F_+^{\xi}$ is spanned by products $u'v'v''u''$, where $v',v''\in F_+^G$ and $u',u''\in F$. Consequently, setting $\nu=\nu(n(\mathfrak{R}))-1$,
we have that $F_+^{\xi \nu}$ is spanned by products
\begin{eqnarray*}\label{monomials in Theorem 2}
w&=&(u_1'v_1'v_1''u_1'')\cdots (u_{\nu}'v_{\nu}'v_{\nu}''u_{\nu}'')\\
&=&u_1'(v_1')(v_1''u_1''u_2'v_2')(v_2''u_2''u_3'v_3')\cdots (v_{\nu-1}''u_{\nu-1}''u_{\nu}'v_{\nu}')(v_{\nu}'')u_{\nu}''
\end{eqnarray*}
(note that by assumption $C\neq 0$, hence $n(\mathfrak{R})\ge 2$ and so $\nu=\nu(n(\mathfrak{R}))-1\ge 2$).
Applying the identity \eqref{Nagata-Higman type identity for R} for
\[
y:=u_1', \quad z:=u_{\nu}'', \quad x_1:=v_1',\quad x_{\nu+1}:=v_{\nu}''\]
and
\[\quad x_{i+1}:=v_{i}''u_i''u_{i+1}'v_{i+1}'\quad \mbox{for}\quad i=1,\ldots,\nu-1,
\]
we obtain that $w$ can be expressed as a linear combination of products of the form
$st$ where  $s$ or $t$ belongs to $\{v_1',v_1'',\dots,v_{\nu}',v_{\nu}''\}$ and both $s$ and $t$ belong to $F_+$.
Therefore we have that $F_+^{\xi \nu}\subseteq F^G_+F_++F_+F^G_+$, implying that
$\gamma(F_+,F^G\otimes (F^G)^{\mathrm{op}})\leq \xi \nu-1$ (recall that the $F^G$-bimodule $F_+$
can be thought of as a left $F^G\otimes (F^G)^{\mathrm{op}}$-module, where  $(F^G)^{\mathrm{op}}$ stands for the opposite ring of $F^G$).
It follows by \eqref{eq:reynolds} that $\gamma(F^G_+,F^G\otimes (F^G)^{\mathrm{op}}) \leq \xi \nu-1$, implying in turn the desired inequality
$\beta(F^G)\leq \xi \nu-1$.
\end{proof}

\begin{theorem}\label{main theorem 3}
Let $G$ be a finite abelian group and
let $\mathfrak R$ be a weakly noetherian variety of associative algebras properly containing the variety of commutative algebras.
Then
\[
\beta(G,{\mathfrak R})\leq (\nu(n(\mathfrak{R}))-1)\vert G\vert(\vert G\vert+1)-1
\]
where $n(\mathfrak{R})$ is the same as  in Theorem \ref{main theorem 1}.
\end{theorem}

\begin{proof}
Let $L$ be an extension of the base field $K$. Embedding the free algebra $K\langle V_{\infty}\rangle$ in
$L\otimes_KK\langle V_{\infty}\rangle\cong L\langle L\otimes_KV_{\infty}\rangle$, we may consider
the variety ${\mathfrak R}_L$ of $L$-algebras defined by the same polynomial identities as the variety of $K$-algebras $\mathfrak R$.
Since the base field $K$ is of characteristic 0, it is well known that
$L\otimes_KF({\mathfrak R},V)\cong F({\mathfrak R}_L,L\otimes_KV)$ is the relatively free algebra over the $L$-vector space $L\otimes_KV$
in the variety ${\mathfrak R}_L$. Similarly, embedding $GL(V)$ in $GL(L\otimes_KV)$, we obtain that
\[
F({\mathfrak R}_L,L\otimes_KV)^G=L\otimes_KF({\mathfrak R},V)^G.
\]
Hence it is sufficient to prove the theorem for $K$ algebraically closed.
Then the finite abelian group $G\subset GL(V)$ is isomorphic to a group of diagonal matrices
and we may assume that $G$ acts on the basis of $V$ as
\[
g(x_i)=\chi_i(g)x_i,\quad \chi_i\in G^{\ast},\quad i=1,\ldots,d.
\]
It follows that $F^G$ is spanned by monomials. Moreover, if $x_{i_1}\cdots x_{i_m}\in F^G$, then $x_{i_{\sigma(1)}}\cdots x_{i_{\sigma(m)}}\in F^G$
for any permutation $\sigma\in S_m$.

It is sufficient to show that any $G$-invariant monomial $w$ of degree at least $(\nu-1)\vert G\vert(\vert G+1\vert)$, $\nu=\nu(n(\mathfrak{R}))$,
can be expressed as a linear combination of  products of invariant monomials of lower degree
(as mentioned in the proof of Theorem~\ref{main theorem 2}, the assumption that $\mathfrak{R}$ properly contains the variety
of commutative algebras implies $\nu\ge 2$).
We claim that any monomial $z$ with $\deg(z)\ge \vert G\vert(\vert G\vert+1)$ belongs to $F(F_+^G)^2F$
(i.e., $z$ contains two consecutive non-trivial $G$-invariant subwords).
Indeed, since by  Lemma \ref{action of abelian groups}  any monomial of degree at least $\vert G\vert$ contains a nontrivial $G$-invariant submonomial, we have
\[z=(u_1'v_1u_1'')(u_2'v_2u_2'')\cdots (u_{\vert G\vert+1}' v_{\vert G\vert+1}u_{\vert G\vert+1}'')\]
where $v_i\in F_+^G$ are non-trivial $G$-invariant monomials and $u_i',u_i''\in F$ are arbitrary monomials.
Apply Lemma~\ref{action of abelian groups} for the words
\[u_1=v_1u_1''u_2', \ u_2=v_2u_2''u_3', \ \dots\ u_{\vert G\vert}=v_{\vert G\vert}u_{\vert G\vert}''u_{\vert G\vert+1}'.\]
We conclude that there exist $1\le i\le j\le \vert G\vert$ such that the monomial $u_iu_{i+1}\cdots u_j\in F_+^G$, hence $z$ contains the subword
$u_i\cdots u_jv_{j+1}\in (F_+^G)^2$, implying $z\in F(F_+^ G)^2F$, so the claim is proved.

Now  take a $G$-invariant monomial
$w\in F({\mathfrak R},V)^G$ of degree at least  $(\nu-1)\vert G\vert(\vert G\vert +1)$.
Write $w$ as a product
\[w=w_1\cdots w_{\nu-1}\]
of monomials $w_i$ where $\deg(w_i)\ge \vert G\vert (\vert G\vert +1)$, hence $w_i\in F(F_+^G)^2F$ for $i=1,\dots,\nu-1$.
So we have
\[w=s_1t_1't_1''s_2t_2't_2''s_3\cdots t_{\nu-1}'t_{\nu-1}''s_{\nu}\]
where $t_i',t_i''\in F_+^G$ and $s_j\in F$.
Apply the identity \eqref{Nagata-Higman type identity for R} for
\[
y:=s_1,\ x_1:=t_1',\ x_2:=t_1''s_2t_2',
\ \dots \ x_{\nu-1}:=t_{\nu-2}''s_{\nu-1}'t_{\nu-1}', \ x_{\nu}:=t_{\nu-1}'',\ z:=s_{\nu}\]
and present $w$ as a linear combination of $G$-invariant monomials $\tilde{w}$ starting or ending by a non-trivial $G$-invariant submonomial $t_i'$ or $t_i''$.
Note that If $\tilde{w}=tu$ or $ut$, where $t\in \{t_i',t_i''\vert i=1,\dots,\nu-1\}$, then $u$ is necessarily
a non-trivial $G$-invariant monomial, so $\tilde{w}\in (F_+^G)^2$, implying in turn that $w\in (F_+^G)^2$.
\end{proof}



\begin{thebibliography}{99}

\bibitem{A}
A. Z. Anan'in,
{\it Locally finitely approximable and locally representable varieties of algebras} (Russian),
Algebra Logika {\bf 16} (1977), 3-23.
Translation: Algebra Logic {\bf 16} (1977), 1-16.

\bibitem{C} K. Cziszter, {\it The Noether number of the non-abelian group of order
$3p$}, Periodica  Math. Hungarica {\bf 68} (2014), 150-159.

\bibitem{CD1}
K. Cziszter, M. Domokos,
{\it On the generalized Davenport constant and the Noether number},
Cent. Eur. J. Math.  {\bf 11}  (2013), No. 9, 1605-1615.

\bibitem{CD2}
K. Cziszter, M. Domokos,
{\it Groups with large Noether bound},
Ann. Inst. Fourier (Grenoble)  {\bf 64}  (2014), No. 3, 909-944.

\bibitem{CD3}
K. Cziszter, M. Domokos,
{\it The Noether number for the groups with a cyclic subgroup of index two},
J. Algebra {\bf 399} (2014), 546-560.

\bibitem{CDG}
K. Cziszter, M. Domokos, A. Geroldinger,
{\it The interplay of invariant theory with multiplicative ideal theory and with arithmetic combinatorics},
Multiplicative Ideal Theory and Factorization Theory, (Scott T. Chapman and M. Fontana and A. Geroldinger and B. Olberding, eds.),
Springer-Verlag, 2016.

\bibitem{DK} H. Derksen, G. Kemper,
{\it Computational Invariant Theory}, Encyclopaedia of Mathematical Sciences {\bf 130},
Springer-Verlag, 2002.

\bibitem{DH}
M. Domokos, P. Heged\H{u}s,
{\it Noether's bound for polynomial invariants of finite groups},
Arch. Math. (Basel)  {\bf 74}  (2000), No. 3, 161-167.

\bibitem{D1}
V. Drensky,
{\it Polynomial identities of finite dimensional algebras},
Serdica {\bf 12} (1986), 209-216.

\bibitem{D2}
V. Drensky,
{\it Finite generation of invariants of finite linear groups on
relatively free algebras},
Lin. and Multilin. Algebra {\bf 35} (1993), 1-10.

\bibitem{D3}
V. Drensky,
{\it Commutative and noncommutative invariant theory},
Math. and Education in Math.,
Proc. of the 24-th Spring Conf. of the Union of Bulgar. Mathematicians,
Svishtov, April 4-7, 1995, Sofia, 1995, 14-50.

\bibitem{DF}
V. Drensky, E. Formanek,
{\it Polynomial Identity Rings},
Advanced Courses in Mathematics, CRM Barcelona, Birkh\"auser Verlag, Basel, 2004.

\bibitem{DuI}
J. Dubnov, V. Ivanov,
{\it Sur l'abaissement du degr\'e des polyn\^omes en affineurs},
C.R. (Doklady) Acad. Sci. USSR {\bf 41} (1943), 96-98
(see also MR {\bf 6} (1945), p. 113, Zbl. f\"ur Math. {\bf 60} (1957), p. 33).

\bibitem{FM}
J. W. Fisher, S. Montgomery,
{\it Invariants of finite cyclic groups acting on generic matrices},
J. Algebra {\bf 99} (1986), 430-437.

\bibitem{F}
E. Formanek,
{\it Noncommutative invariant theory},
Contemp. Math. {\bf 43} (1985), 87-119.

\bibitem{GZ}
A. Giambruno, M. Zaicev,
{\it Polynomial Identities and Asymptotic Methods},
Mathematical Surveys and Monographs {\bf 122}, AMS, Providence, RI, 2005.

\bibitem{H}
G. Higman,
{\it On a conjecture of Nagata},
Proc. Camb. Philos. Soc. {\bf 52} (1956), 1-4.

\bibitem{K1}
A. R. Kemer,
{\it Nonmatrix varieties} (Russian),
Algebra Logika {\bf 19} (1980), 255-283.
Translation: Algebra Logic {\bf 19} (1981), 157-178.

\bibitem{K2}
A. R. Kemer,
{\it Asymptotic basis of the identities with unit of the variety $\text{\rm Var}(M_2(F))$}  (Russian),
Izv. Vyssh. Uchebn. Zaved. Mat. (1989), No. 6, 71-76.
Translation: Sov. Math. Soviet Math. (Iz. VUZ) {\bf 33} (1990), No. 6, 71-76.

\bibitem{Kh}
V. K. Kharchenko,
{\it Noncommutative invariants of finite groups and Noetherian varieties},
J. Pure Appl. Alg. {\bf 31} (1984), 83-90.

\bibitem{KS}
O. G. Kharlampovich, M. V. Sapir,
{\it Algorithmic problems in varieties},
Intern. J. Algebra and Computation {\bf 5} (1995), 379-602.

\bibitem{Ku}
E. N. Kuz'min,
{\it On the Nagata-Higman theorem} (Russian),
Mathematical Structures, Computational Mathematics,
Mathematical Modelling. Proc. Deducated to the 60th Birthday of
Acad. L. Iliev, Sofia, 1975, 101-107.

\bibitem{La}
V. N. Latyshev,
{\it Generalization of the Hilbert theorem on the finiteness of bases} (Russian),
Sib. Mat. Zhurn. {\bf 7} (1966), 1422-1424.
Translation: Sib. Math. J. {\bf 7} (1966), 1112-1113.

\bibitem{Lv}
I. V. Lvov,
{\it Maximality conditions in algebras with identity relations} (Russian),
Algebra Logika {\bf 8} (1969), 449-459.
Translation: Algebra Logic {\bf 8} (1969), 258-263.

\bibitem{M}
A. I. Malcev,
{\it On the representations of infinite algebras} (Russian),
Mat. Sb. {\bf 13} (1943), 263-286.

\bibitem{Na}
M. Nagata,
{\it On the nilpotency of nil algebras},
J. Math. Soc. Japan {\bf 4} (1953), 296-301.

\bibitem{N}
E. Noether,
{\it Der Endlichkeitssatz der Invarianten endlicher Gruppen},
Math. Ann. {\bf 77} (1916), 89-92;
reprinted in ``Gesammelte Abhandlungen. Collected Papers'',
Springer-Verlag, Berlin-Heidelberg-New York-Tokyo, 1983, 181-184.

\bibitem{R}
Yu. P. Razmyslov,
{\it Trace identities of full matrix algebras over a field of
characteristic zero} (Russian),
Izv. Akad. Nauk SSSR, Ser. Mat. {\bf 38} (1974), 723-756.
Translation: Math. USSR Izv. {\bf 8} (1974), 723-760.

\bibitem{Sch}
B. J. Schmid,
{\it Finite groups and invariant theory},
Topics in Invariant Theory (Paris, 1989/1990),
Lecture Notes in Mathematics, {\bf 1478}, Springer-Verlag, Berlin, 1991, 35-66.

\bibitem{T}
I. K. Tonov,
{\it Just-non-weakly noetherian varieties of associative algebras} (Russian),
Pliska Stud. Math. Bulg. {\bf 2}, 1981, 162-166.

\bibitem{Z}
E. I. Zelmanov,
{\it On Engel Lie algebras} (Russian),
Sibirsk. Mat. Zh. {\bf 29} (1988), No. 5, 112-117.
Translation: Sib. Math. J. {\bf 29} (1988), 777-781.

\end{thebibliography}
\end{document}